\documentclass[11pt,a4paper]{article}

\usepackage[T1]{fontenc}
\usepackage[utf8]{inputenc}
\usepackage{lmodern}
\usepackage{amsmath,amssymb,amsthm,mathtools,mathrsfs}
\usepackage{array,booktabs}
\usepackage{geometry}
\usepackage{microtype}
\usepackage{hyperref}

\hypersetup{
    hidelinks,
    pdftitle={Support profiles of full-capacity Pareto spectra of order three},
    pdfauthor={Samir Adly}
}
\allowdisplaybreaks

\newcommand{\R}{\mathbb{R}}
\newcommand{\diag}{\operatorname{diag}}
\newcommand{\adj}{\operatorname{adj}}
\newcommand{\Res}{\operatorname{Res}}
\newcommand{\rhoSp}{\rho}
\newcommand{\Ind}{\mathcal I}
\newcommand{\Supp}{\mathscr S}
\newcommand{\Id}{\operatorname{Id}}

\numberwithin{equation}{section}

\theoremstyle{plain}
\newtheorem{theorem}{Theorem}[section]
\newtheorem{proposition}[theorem]{Proposition}
\newtheorem{lemma}[theorem]{Lemma}
\newtheorem{corollary}[theorem]{Corollary}

\renewenvironment{abstract}
  {\par\smallskip\small\noindent\textbf{Abstract. }\ignorespaces}
  {\par\medskip}

\theoremstyle{definition}
\newtheorem{definition}[theorem]{Definition}

\theoremstyle{remark}
\newtheorem{remark}[theorem]{Remark}

\title{Support profiles of full-capacity Pareto spectra of order three}
\author{
Samir Adly\thanks{Laboratoire XLIM, Université de Limoges, 123 avenue Albert Thomas, 87060 Limoges CEDEX, France.\hskip 5mm
Email: samir.adly@unilim.fr, \url{https://www.unilim.fr/pages_perso/samir.adly/}.}
}
\date{}
\begin{document}
\maketitle

\begin{abstract}
For a given real matrix $A\in\R^{n\times n}$ of order $n\geq 1$, a Pareto eigenvalue is a scalar
$\lambda\in\R$ for which there exists a nonzero vector $x\in\R^n_+$
such that $Ax-\lambda x\in\R^n_+$ and
$\langle x,Ax-\lambda x\rangle=0$.
It is known that a real matrix of order $n=3$ has at most nine distinct
Pareto eigenvalues. The main contribution of this paper is to classify
all possible distributions of these nine values among their producing
supports when this maximum is attained. In general, a Pareto eigenvalue may be produced by several supports.
We show that, in order three, there are at most nine pairs $(J,\lambda)$
such that the support $J$ produces $\lambda$. Thus, at full capacity,
each of the nine values has a unique producing support. The numbers of
values produced on supports of sizes $1$, $2$, and $3$, respectively,
form necessarily one of the following triples
$$
    (1,5,3),\qquad (2,4,3),\qquad \text{or}\qquad (2,5,2).
$$
Each triple determines the counts on the individual supports up to a
simultaneous permutation of the indices. In the proof, we show that
the pairs producing two distinct values form a triangle-free graph.
Together with the restrictions on singleton and full supports, this
also gives an independent proof that the Pareto capacity in order
$n=3$ is nine.
For each of the three profiles, we give an explicit matrix with nine
distinct regular Pareto eigenvalues. We prove that the profile is
locally constant at every such matrix, so each profile occurs on a
nonempty Euclidean-open set. We also exhibit a path of full-capacity
matrices along which the profile changes.
\end{abstract}



\section{Introduction}

Let $A=(a_{ij})\in\R^{n\times n}$. Vector inequalities are understood
componentwise. A real number $\lambda\in\R$ is a \emph{Pareto eigenvalue}
of $A$ if there is a vector $x\in\R^n\setminus\{0\}$ such that
\begin{equation}\label{eq:eicp}
    x\geq0,
    \qquad
    Ax-\lambda x\geq0,
    \mbox{ and }
    \langle x,Ax-\lambda x\rangle=0.
\end{equation}
The set of Pareto eigenvalues of $A$ is denoted by $\Pi(A)$. The
\emph{Pareto capacity} of order $n$ is
$$
    c_n:=\max_{A\in\R^{n\times n}}|\Pi(A)|.
$$
A matrix $A\in\R^{n\times n}$ is called \emph{full-capacity} if
$|\Pi(A)|=c_n$, where $|E|$ denotes the cardinality of a finite set
$E$.

The support description of Pareto eigenvalues was given by
Seeger~\cite[Theorem~4.1]{Seeger1999}.
The value $c_2=3$ is classical; see
\cite[Proposition~16]{QueirozJudiceHumes2004}. A complete description
of the Pareto spectra of real matrices of order $2$ was given in
\cite[Example~2 and Tables~1 and~2]{PintoDaCostaSeeger2010}.
The same paper introduced spectral histograms and their aggregated
versions \cite[Definition~3]{PintoDaCostaSeeger2010}. In
Proposition~\ref{prop:order-two-profiles}, we record the $2$
full-capacity profiles in the assigned-profile terminology used here.

Baillon and Seeger \cite[Theorem~2]{BaillonSeeger2021} proved
that $c_3=9$. Seeger and
Vicente-P\'erez~\cite{SeegerVicente2011} gave a family of matrices of
order $n\geq2$ depending on a parameter $s$. Each matrix in this family has exactly
$3(2^{n-1}-1)$ Pareto eigenvalues under the condition that $s$ is
transcendental and $s>1+\sqrt{2}$, which means that
$c_n\geq 3(2^{n-1}-1)$. We shall use its order $n=3$ case in
Section~\ref{sec:realizations}.
Baillon and Seeger introduced regular Pareto eigenvalues and the
regular Pareto capacity $c_n^{\rm reg}$ in
\cite{BaillonSeeger2020}. They also proved supportwise lower
stability: if a value is regularly produced by a fixed support, then,
under a small perturbation, a nearby regular value is produced by the
same support \cite[Proposition~3]{BaillonSeeger2020}.

Several numerical methods have been developed for cone-constrained
eigenvalue problems, including semismooth Newton methods
\cite{AdlySeeger2011}, the Lattice Projection Method
\cite{AdlyRammal2013}, and interior-point methods
\cite{AdlyHaddouLe2023}. These works also contain numerical experiments
on matrices with large Pareto spectra.

Baillon and Seeger also considered spectrally separable matrices,
for which each Pareto eigenvalue has a unique producing support, and
noted that these matrices form a dense set
\cite[Section~1.1]{BaillonSeeger2021}. Section~\ref{sec:stability} compares our local constancy result with Kielstra's.

This paper concerns the way in which the $9$ Pareto eigenvalues
of a full-capacity $3\times3$ matrix are produced. If $\Pi_J(A)$
denotes the set of values produced by a support $J$, then a Pareto
eigenvalue produced by $J$ is an eigenvalue of the principal submatrix
$A_J$, subject to sign conditions on an eigenvector and on the rows
outside $J$ (see \cite[Theorem~4.1]{Seeger1999}). For a general matrix, a value may occur on several
supports. Writing $N(A)=\sum_{J\ne\varnothing}|\Pi_J(A)|$, where $J$
ranges over the nonempty subsets of $\{1,\ldots,n\}$, we prove that
$|\Pi(A)|\leq N(A)\leq9$ when $n=3$. At full capacity there is equality
throughout. Thus every value has a unique producing support.

Three restrictions yield the upper bound. The pairs producing two
distinct values form a triangle-free graph. If all three singletons produce,
the full support produces at most one value. If two singletons produce and
each forms a pair producing two values with the remaining index, the full
support produces at most two. An explicit matrix attaining nine distinct
values then gives an independent proof of $c_3=9$.

Pinto da Costa and Seeger~\cite{PintoDaCostaSeeger2010} give two explicit $3\times3$ matrices
with 9 distinct Pareto eigenvalues and report their spectra in
Tables~4 and~8 of \cite{PintoDaCostaSeeger2010}. On p.~48, they give
the corresponding spectral histograms, whose aggregated versions are
$(2,5,2)$ and $(2,4,3)$, respectively. In the present paper, we
classify the aggregate and detailed support profiles of all real
$3\times3$ matrices with 9 distinct Pareto eigenvalues. More
precisely, for every such matrix, the counts on supports of sizes
$1$, $2$, and $3$ form one of the triples
$$
    (1,5,3),\qquad (2,4,3),\qquad\text{or}\qquad (2,5,2).
$$
Among the other aggregate profiles compatible with the trivial
supportwise bounds, we show that
$$
    (0,6,3),\quad (1,6,2),\quad (2,6,1),\quad
    (3,3,3),\quad (3,4,2),\quad (3,5,1),
    \mbox{ and } (3,6,0)
$$
are impossible.

Moreover, up to a simultaneous permutation of the indices, the
corresponding detailed profiles are
$$
    (1,0,0;1,2,2;3),\qquad
    (1,1,0;1,2,1;3),\qquad
    (1,1,0;1,2,2;2).
$$
Thus, although the equality $c_3=9$ is already known, our result gives, to the best of our knowledge, the first complete
classification of how the $9$ distinct Pareto eigenvalues of a
full-capacity matrix can be distributed among their producing supports.

We show that the classification is sharp by giving exact regular
realizations of all $3$ profiles. For the profile $(2,5,2)$, we use
the order $n=3$ case of the family introduced in
\cite[Eq.~(2.5)]{SeegerVicente2011}. We determine its detailed profile and prove that all 9
values are regular for every real $s>1+\sqrt{2}$.

At any full-capacity matrix with 9 regular values, supportwise
stability makes the profile locally constant. The set of all such
matrices is therefore the disjoint union of three nonempty open sets
indexed by aggregate profile. An independent algebraic criterion
separates the spectra of distinct principal submatrices. The family
realizing $(2,5,2)$ has 9 regular values even though two principal
submatrices share an ordinary eigenvalue.

Section~\ref{sec:definitions} gives the definitions and basic facts.
The restrictions used in the proof are established in
Section~\ref{sec:structure}. Section~\ref{sec:capacity} proves the bound on
$N(A)$ and uniqueness at full capacity. Section~\ref{sec:classification}
classifies the profiles, and Section~\ref{sec:realizations} gives their
regular realizations. Section~\ref{sec:stability} establishes local
constancy and spectral separation, and describes a transition between
two profiles at full capacity.

\section{Supports, assignments, and regularity}
\label{sec:definitions}

For $n\geq1$, set
$$
    \Ind_n:=\{1,\ldots,n\},
    \qquad
    \Supp_n:=\{J\subseteq\Ind_n:J\neq\varnothing\}.
$$
For $J\in\Supp_n$, let $A_J$ be the principal submatrix of $A$ indexed
by $J$. We also write $J^{\mathrm c}:=\Ind_n\setminus J$ and denote by
$A_{J^{\mathrm c},J}$ the submatrix with rows in $J^{\mathrm c}$ and
columns in $J$.

A support $J$ \emph{produces} $\lambda$ if there is a vector
$u\in\R^{|J|}$ such that
\begin{equation}\label{eq:support-criterion}
    A_Ju=\lambda u,
    \qquad
    u>0,
    \qquad
    A_{J^{\mathrm c},J}u\geq0.
\end{equation}
Let $\Pi_J(A)$ be the set of values produced by $J$, and set
\begin{equation}\label{NumberPJ}
    p_J(A):=|\Pi_J(A)|.
\end{equation}
The support criterion gives
\begin{equation}\label{eq:support-union}
    \Pi(A)=\bigcup_{J\in\Supp_n}\Pi_J(A);
\end{equation}
see \cite[Theorem~4.1]{Seeger1999} and \cite[Section~2]{BaillonSeeger2020}. The same value may occur in several
sets $\Pi_J(A)$.

For $1\leq k\leq n$, let
$$
    P_k(A):=\sum_{\substack{J\in\Supp_n\\ |J|=k}}p_J(A),
    \qquad N(A):=\sum_{k=1}^n P_k(A).
$$
Thus $P_k(A)$ counts the pairs $(J,\lambda)$ with $|J|=k$ and
$\lambda\in\Pi_J(A)$. In the total $N(A)$, a value produced by several
supports is counted once for each of them.

We first introduce the supportwise notion of regularity used throughout
the paper.
\begin{definition}\label{def:regular}
A value $\lambda\in\Pi_J(A)$ is \emph{regularly produced by $J$} if it
is an algebraically simple eigenvalue of $A_J$ and every component of
$A_{J^{\mathrm c},J}u$ in \eqref{eq:support-criterion} is positive.
For $J=\Ind_n$, the second condition is empty. A Pareto eigenvalue is
called \emph{regular} if it is regularly produced by at least one
support.
\end{definition}

The notion of a regular Pareto eigenvalue was introduced by Baillon
and Seeger in \cite[Definition~1(iii)]{BaillonSeeger2020}. The
terminology ``regularly produced by $J$'' used here is the supportwise
version of that definition. It records explicitly that algebraic
simplicity and strictness hold with respect to the same producing
support $J$.

We now introduce the assignment used to count each distinct Pareto
eigenvalue only once.

\begin{definition}\label{def:assignment}
Let $A\in\R^{n\times n}$. An \emph{assigned support map} is a map
$$
    \tau:\Pi(A)\longrightarrow\Supp_n
$$
such that $\lambda\in\Pi_{\tau(\lambda)}(A)$ for every
$\lambda\in\Pi(A)$. For $J\in\Supp_n$, define
$$
    m_J(A,\tau)
    :=\bigl|\{\lambda\in\Pi(A):\tau(\lambda)=J\}\bigr|.
$$
For $1\leq k\leq n$, define
$$
    S_k(A,\tau):=
    \sum_{\substack{J\in\Supp_n\\ |J|=k}}m_J(A,\tau).
$$
\end{definition}

Such a map exists by \eqref{eq:support-union}. {The nonempty members of the family
$$
  \{\lambda\in\Pi(A):\tau(\lambda)=J\},
  \qquad J\in\Supp_n,
$$
form a partition of $\Pi(A)$. Thus the numbers $m_J(A,\tau)$ count each
distinct Pareto eigenvalue exactly once, whereas the numbers $p_J(A)$
may count the same value for several supports.}

The assigned profile may depend on the chosen support map, as the next
example shows.
\begin{remark}\label{rem:assignment-example}
Let $A=\Id_3$, the identity matrix of order $n=3$. Every nonempty support $J\subseteq\Ind_3$ produces the
value $1$, because $A_Ju=u$ for every $u>0$ and the outside slack is
zero. Thus $p_J(A)=1$ for all seven supports, but $\Pi(A)=\{1\}$.
In particular, $(P_1(A),P_2(A),P_3(A))=(3,3,1)$ and $N(A)=7$.
The value $1$ may be assigned to any support. Assigning it to $\{1\}$
gives the detailed profile $(1,0,0;0,0,0;0)$, while assigning it to
$\Ind_3$ gives $(0,0,0;0,0,0;1)$. Hence the assigned profile may
depend on $\tau$.
\end{remark}

Unless stated otherwise, a result about an assigned profile is meant
to hold for every assigned support map. In order three,
Theorem~\ref{thm:c3} will show that this map is unique at full
capacity. Proposition~\ref{prop:generic-uniqueness} gives a separate
sufficient condition in terms of the principal spectra.

By definition,
\begin{equation}\label{eq:profile-sum}
    |\Pi(A)|=\sum_{k=1}^nS_k(A,\tau).
\end{equation}
Since the values assigned to $J$ are distinct eigenvalues of $A_J$,
\begin{equation}\label{eq:support-count-bound}
    0\leq m_J(A,\tau)\leq p_J(A)\leq |J|.
\end{equation}
Summing over supports of size $k$ gives $S_k(A,\tau)\leq P_k(A)$.
By \eqref{eq:profile-sum}, summing over $k$ then gives $|\Pi(A)|\leq N(A)$.

For $n=3$, we use the notation
$$
    m_i:=m_{\{i\}}(A,\tau),
    \qquad
    m_{ij}:=m_{\{i,j\}}(A,\tau),
    \qquad
    m_{123}:=m_{\Ind_3}(A,\tau).
$$
The \emph{detailed assigned profile} is
$$
    (m_1,m_2,m_3;m_{12},m_{13},m_{23};m_{123}),
$$
and the \emph{aggregate assigned profile} is $(S_1,S_2,S_3)$, with
{$$
    S_1=m_1+m_2+m_3,\;
    S_2=m_{12}+m_{13}+m_{23}, \mbox{ and }
    S_3=m_{123}.
$$
}

The next proposition describes several transformations that preserve
the producing supports.
\begin{proposition}
\label{prop:invariance}
Let $P$ be a permutation matrix, let $D$ be a positive diagonal
matrix, let $\alpha>0$, and let $\beta\in\R$. If
$$
    B=\alpha D^{-1}P^\top APD+\beta\Id_n,
$$
then
$$
    \Pi(B)=\alpha\Pi(A)+\beta.
$$
The producing supports are relabelled by the permutation associated
with $P$, and regular production is preserved.
\end{proposition}
\begin{proof}
We consider the transformations separately. Permutation similarity only
relabels the coordinates and hence the producing supports. It also
preserves the algebraic multiplicities of the eigenvalues of the
corresponding principal submatrices.

Suppose next that
$$
    B=D^{-1}AD,
$$
and that a support $J$ produces $\lambda$ for $A$ through a vector
$u>0$. Since
$$
    B_J=D_J^{-1}A_JD_J,
    \qquad
    B_{J^{\mathrm c},J}
    =D_{J^{\mathrm c}}^{-1}A_{J^{\mathrm c},J}D_J,
$$
the vector
$$
    v=D_J^{-1}u>0
$$
satisfies
$$
    B_Jv=\lambda v,
    \qquad
    B_{J^{\mathrm c},J}v
    =D_{J^{\mathrm c}}^{-1}A_{J^{\mathrm c},J}u\geq0.
$$
Thus $J$ produces $\lambda$ for $B$. Since $D_{J^{\mathrm c}}$ is
positive diagonal, strict outside inequalities are also preserved.
Moreover, $B_J$ is similar to $A_J$, so algebraic simplicity is
preserved.

Finally, suppose that
$$
    B=\alpha A+\beta\Id_n,
    \qquad \alpha>0.
$$
If $J$ produces $\lambda$ for $A$ through $u>0$, then
$$
    B_Ju=(\alpha\lambda+\beta)u,
    \qquad
    B_{J^{\mathrm c},J}u
    =\alpha A_{J^{\mathrm c},J}u\geq0.
$$
Hence $J$ produces $\alpha\lambda+\beta$ for $B$. Strict inequalities
are preserved because $\alpha>0$, and algebraic simplicity is
preserved under the affine transformation
$\lambda\mapsto\alpha\lambda+\beta$.\\
Applying the inverse permutation, the inverse positive diagonal
similarity, and the inverse affine transformation gives the reverse
inclusion. Therefore,
$$
    \Pi(B)=\alpha\Pi(A)+\beta,
$$
with the producing supports relabelled by the permutation associated
with $P$, and regular production is preserved.
\end{proof}
\begin{remark}
The spectral identities associated with permutation similarity,
positive scaling, and scalar translation appear in
\cite[Proposition~2]{PintoDaCostaSeeger2010}. Proposition~\ref{prop:invariance}
also treats positive diagonal similarity and describes the effect of
these transformations at the level of producing supports. In particular,
permutation similarity relabels the supports, while positive diagonal
similarity, positive scaling, and scalar translation preserve the
support. The proposition also shows that regular production is preserved.
\end{remark}
The equality $c_2=3$ is known; see \cite[Proposition~16]{QueirozJudiceHumes2004} and
\cite[Proposition~9]{PintoDaCostaSeeger2010}. We include a short proof
because the same calculation will be used below.

\begin{proposition}
\label{prop:order-two-profiles}
The Pareto capacity in order $2$ is
$$
    c_2=3.
$$
Let $A\in\R^{2\times2}$ be full-capacity, and let $\tau$ be an
assigned support map. Up to an interchange of the indices, the detailed
assigned profile is either
$$
    (1,1;1)
    \qquad\text{or}\qquad
    (1,0;2).
$$
The corresponding aggregate profiles are
$$
    (2,1)
    \qquad\text{and}\qquad
    (1,2),
$$
respectively. In both cases, every Pareto eigenvalue has a unique
producing support and is regularly produced. Both profiles occur.
\end{proposition}

\begin{proof}
Let $
    A=
    \begin{pmatrix}
        a&b\\
        c&d
    \end{pmatrix}.
$
Each singleton support produces at most $1$ value, and the full support
produces at most $2$ values.

Suppose first that the full support produces $2$ distinct values. A
positive eigenvector can be written as $(1,t)^\top$ with $t>0$.
Eliminating the eigenvalue gives
$$
    bt^2+(a-d)t-c=0.
$$
The first row gives $\lambda=a+bt$, so the two distinct eigenvalues
force $b\ne0$ and give two distinct positive roots. Their product
is $-c/b>0$, so
$$
    bc<0.
$$
The support $\{1\}$ can produce only if $c\geq0$, while the support
$\{2\}$ can produce only if $b\geq0$. Hence at most $1$ singleton
support produces. It follows that
$$
    |\Pi(A)|\leq3.
$$

If the full support produces at most $1$ value, then the same bound
follows from the fact that the $2$ singleton supports produce at most
$1$ value each. Therefore
$$
    c_2\leq3.
$$
The matrix
$
    \begin{pmatrix}
        3&-1\\
        2&0
    \end{pmatrix}
$
has the $3$ distinct Pareto eigenvalues $1$, $2$, and $3$. Hence
$c_2=3$.

Now let $A$ be full-capacity. Suppose that the full support produces
$2$ values. The preceding argument gives $bc<0$. Exactly $1$ of the
singleton supports then produces. Up to an interchange of the indices,
the detailed profile is
$$
    (1,0;2).
$$
Moreover, $bc\neq0$, so neither $a$ nor $d$ is an eigenvalue of $A$.
Thus no value is produced by both a singleton support and the full
support.

Suppose next that the full support produces at most $1$ value. Full
capacity implies that it produces exactly $1$ value and that both
singleton supports produce. Hence
$$
    b\geq0,
    \qquad
    c\geq0.
$$
In fact, $bc>0$. Indeed, if $bc=0$, then the eigenvalues of $A$ are
$a$ and $d$, and the full support cannot produce a third distinct
value. Therefore the detailed profile is
$$
    (1,1;1).
$$
Full capacity also gives $a\neq d$. The value produced by the full
support is different from $a$ and $d$. Hence each of the $3$ Pareto
eigenvalues has a unique producing support.

In the profile $(1,0;2)$, the outside slack of the producing singleton
is strictly positive, and the $2$ full-support eigenvalues are distinct.
In the profile $(1,1;1)$, we have $b,c>0$, so both singleton
productions are strict. Also,
$$
    (a-d)^2+4bc>0,
$$
so the value produced by the full support is a simple eigenvalue.
Thus all $3$ values are regularly produced.

Finally,
$$
    \begin{pmatrix}
        3&-1\\
        2&0
    \end{pmatrix}
    \qquad\text{and}\qquad
    \begin{pmatrix}
        0&1\\
        1&2
    \end{pmatrix}
$$
realize the profiles $(1,0;2)$ and $(1,1;1)$, respectively.
\end{proof}
\begin{remark}
The complete Pareto spectra of real matrices of order $2$ were
described in \cite[Example~2 and Tables~1 and~2]
{PintoDaCostaSeeger2010}. The same paper introduced the spectral
histogram and its aggregated version in
\cite[Definition~3]{PintoDaCostaSeeger2010}. The proposition above
records the full-capacity case in the assigned-profile terminology used
here. In order $2$, every Pareto eigenvalue of a full-capacity matrix
has a unique producing support. Hence no assignment ambiguity occurs.
\end{remark}

\section{Rich pair supports and order-three restrictions}
\label{sec:structure}

A pair support is called \emph{rich} if it produces two distinct
Pareto eigenvalues. The next lemma gives a hereditary bound from the order-$2$ case.

\begin{lemma}\label{lem:heredity-two}
Let $A\in\R^{3\times3}$ and let $\tau$ be an assigned support map. For
every pair $\{i,j\}\subseteq\Ind_3$, we have
\begin{equation}\label{eq:pair-heredity}
    m_i+m_j+m_{ij}\leq3.
\end{equation}
Therefore,
\begin{equation}\label{eq:2S1S2}
    2S_1+S_2\leq9.
\end{equation}
\end{lemma}

\begin{proof}
Every value assigned to $\{i\}$, $\{j\}$, or $\{i,j\}$ is a Pareto
eigenvalue of the principal submatrix $A_{\{i,j\}}$. Indeed, passing to
this principal submatrix only removes outside slack inequalities.
Since $\tau$ assigns each value to a single support, the sum
$m_i+m_j+m_{ij}$ counts distinct Pareto eigenvalues of $A_{\{i,j\}}$. Thus
\eqref{eq:pair-heredity} follows from Proposition~\ref{prop:order-two-profiles}.
Summing it over the $3$ pairs gives \eqref{eq:2S1S2}.
\end{proof}
The following lemma gives the sign pattern imposed by a rich pair support.
\begin{lemma}\label{lem:rich-pair}
If the pair support $\{i,j\}$ is rich, then
\begin{equation}\label{eq:rich-sign}
    a_{ij}a_{ji}<0\quad \mbox{ and }\quad
    \frac{a_{jj}-a_{ii}}{a_{ij}}>0.
\end{equation}
In particular, if $a_{ii}<a_{jj}$, then $a_{ij}>0$ and $a_{ji}<0$.
\end{lemma}

\begin{proof}
By \eqref{eq:support-criterion}, each of the two distinct values
produced by $J=\{i,j\}$ in $A$ is an eigenvalue of $A_J$ with a strictly
positive eigenvector $u$. The outside slack inequalities
$A_{J^{\mathrm c},J}u\geq0$ can only rule out candidate eigenpairs of $A_J$.
A positive eigenvector of $A_{\{i,j\}}$ can be written as
$(1,t)^\top$ with $t>0$. Eliminating the eigenvalue gives
$$
    a_{ij}t^2+(a_{ii}-a_{jj})t-a_{ji}=0.
$$
Since $\lambda=a_{ii}+a_{ij}t$, two distinct produced values
force $a_{ij}\ne0$ and give two distinct positive roots. Their product
and their sum are positive, which gives \eqref{eq:rich-sign}.
\end{proof}
We now show that rich pair supports cannot form a triangle.
\begin{theorem}
\label{thm:triangle-free}
Let $n\geq3$ and $A\in\R^{n\times n}$. Let $G_2(A)$ be the graph with
vertex set $\Ind_n$ whose edges are the rich pair supports. Then
$G_2(A)$ is triangle-free.
\end{theorem}

\begin{proof}
Suppose that $\{i,j,k\}$ is a triangle in $G_2(A)$. By
Lemma~\ref{lem:rich-pair}, the $3$ diagonal entries are distinct.
Relabel the indices so that $
    a_{ii}<a_{jj}<a_{kk}.$
The sign rule then gives
$$
    a_{ij},a_{ik},a_{jk}>0 \quad \mbox{ and }\quad
    a_{ji},a_{ki},a_{kj}<0.
$$
For a positive vector $(u_i,u_j)^\top$ on $\{i,j\}$, the outside
slack in row $k$ is
$$
    a_{ki}u_i+a_{kj}u_j<0.
$$
Thus $\{i,j\}$ cannot produce a Pareto eigenvalue. This contradicts
the assumption that it is rich.
\end{proof}
The triangle-free property gives the following global bound on pair supports.
\begin{corollary}\label{cor:pair-bound}
Let $A\in\R^{n\times n}$ with $n\geq2$, and let $\tau$ be an assigned support map.
Then
\begin{equation}\label{eq:pair-global}
    S_2(A,\tau) {}\leq P_2(A)
    \leq \binom n2+\left\lfloor\frac{n^2}{4}\right\rfloor.
\end{equation}
In order $n=3$, we have $S_2\leq P_2(A)\leq5$.
\end{corollary}

\begin{proof}
For every pair support $J$,
$$
    m_J(A,\tau)\leq p_J(A)
    \leq1+\mathbf{1}_{\{J\text{ is rich}\}}.
$$
Hence
$$
    S_2(A,\tau) {}\leq P_2(A)\leq\binom n2+|E(G_2(A))|.
$$
Write $e:=|E(G_2(A))|$, and let $d_i$ be the degree of vertex $i$.
The graph is triangle-free by Theorem~\ref{thm:triangle-free} for
$n\geq3$, and trivially so for $n=2$. Hence the endpoints of each edge have
disjoint neighbor sets, so $d_i+d_j\leq n$ for every edge $\{i,j\}$.
Therefore
$$
    \sum_{i=1}^n d_i^2
    =\sum_{\{i,j\}\in E(G_2(A))}(d_i+d_j)
    \leq ne.
$$
Since $\sum_i d_i=2e$, the Cauchy-Schwarz inequality gives
$$
    4e^2\leq n\sum_{i=1}^n d_i^2\leq n^2e.
$$
It follows that $e\leq\lfloor n^2/4\rfloor$, since $e$ is a
nonnegative integer. In order three, at most two of the three pairs
can be rich. Each rich pair produces two values and each remaining
pair produces at most one, so $P_2(A)\leq 2+2+1=5$ before any values
are assigned.
\end{proof}
The next lemma restricts the contribution of the full support when all singleton supports produce.
\begin{lemma}\label{lem:all-singletons}
Let $A\in\R^{3\times3}$. If all $3$ singleton supports produce
Pareto eigenvalues, then
$$
    |\Pi_{\Ind_3}(A)|\leq1.
$$
\end{lemma}

\begin{proof}
If all $3$ singleton supports produce, then every off-diagonal entry
of $A$ is nonnegative. Choose $\gamma>0$ such that
$B:=A+\gamma\Id_3$ is entrywise nonnegative. Let $x>0$ satisfy
$Ax=\lambda x$. Set
$$
    \mu:=\lambda+\gamma,
    \qquad
    D:=\diag(x_1,x_2,x_3),
    \qquad
    C:=D^{-1}BD.
$$
Then $C\geq0$ and $C\mathbf{1}=\mu\mathbf{1}$. Thus every row sum of
$C$ is $\mu$, so $\mu\geq0$ and $\|C\|_\infty=\mu$. Since $\mu$ is
an eigenvalue of $C$,
$$
    \mu\leq\rhoSp(C)\leq\|C\|_\infty=\mu.
$$
It follows that $\mu=\rhoSp(C)=\rhoSp(B)$. Therefore every value
produced by the full support is equal to $\rhoSp(B)-\gamma$, and there
is at most one such value.
\end{proof}
The next lemma describes the obstruction created by a singleton support adjacent to two rich pairs.
\begin{lemma}\label{lem:active-center}
Let $\{i,j,k\}=\Ind_3$. Suppose that $\{i\}$ produces a Pareto
eigenvalue and that $\{i,j\}$ and $\{i,k\}$ are rich. Then
$\{j,k\}$ produces no Pareto eigenvalue.
\end{lemma}

\begin{proof}
After relabeling, take $i=1$, $j=2$, and $k=3$. Production by $\{1\}$
gives $a_{21},a_{31}\geq0$. Since $\{1,2\}$ is rich,
Lemma~\ref{lem:rich-pair} gives $a_{21}>0$ and $a_{12}<0$. In the same
way, $a_{31}>0$ and $a_{13}<0$. For every
$(u_2,u_3)^\top>0$, the outside slack of $\{2,3\}$ in row $1$ is
$$
    a_{12}u_2+a_{13}u_3<0.
$$
Hence $\{2,3\}$ produces no Pareto eigenvalue.
\end{proof}
The next lemma bounds the contribution of the full support when two singleton supports and two adjacent pair supports produce.
\begin{lemma}\label{lem:interval-obstruction}
Suppose that the singleton supports $\{1\}$ and $\{2\}$ produce
Pareto eigenvalues and that the pair supports $\{1,3\}$ and
$\{2,3\}$ are rich. Then
$$
    |\Pi_{\Ind_3}(A)|\leq2.
$$
\end{lemma}

\begin{proof}
The two singleton productions give
$$
    a_{21},a_{31},a_{12},a_{32}\geq0.
$$
By Lemma~\ref{lem:rich-pair}, the richness assumptions imply
$$
    a_{31},a_{32}>0,
    \qquad
    a_{13},a_{23}<0,
    \qquad
    a_{11},a_{22}>a_{33}.
$$
Subtract $a_{33}\Id_3$ and then use the positive diagonal similarity
with
$$
    D:=\diag(-a_{13},-a_{23},1).
$$
The entries in positions $(1,3)$ and $(2,3)$ become $-1$. The other
signs stated above are unchanged. By
Proposition~\ref{prop:invariance}, it is enough to consider
\begin{equation}\label{eq:normal-form}
    A=
    \begin{pmatrix}
        a&p&-1\\
        r&b&-1\\
        u&v&0
    \end{pmatrix},
    \qquad
    a,b,u,v>0,
    \quad p,r\geq0.
\end{equation}

For the support $\{1,3\}$, write a positive eigenvector as
$(1,z)^\top$. The eigenvalue equations and the outside slack condition
become
$$
    z^2-az+u=0,
    \qquad
    r-z\geq0.
$$
Since the pair is rich, the quadratic has two distinct positive roots
$z_-<z_+$, and both satisfy $z\leq r$. Thus $r\geq z_+$. Since the leading coefficient of the quadratic is $1$ and $r\geq z_+$, we have
\begin{equation}\label{eq:F-nonnegative}
F:=r^2-ar+u=(r-z_-)(r-z_+)\geq0.
\end{equation}
For the support $\{2,3\}$, the same argument gives
\begin{equation}\label{eq:G-nonnegative}
    G:=p^2-bp+v\geq0.
\end{equation}

Set
$$
    \theta_1:=b-p,
    \qquad
    \theta_2:=a-r,
    \qquad
    \chi_A(t):=\det(t\Id_3-A).
$$
Suppose that $t$ is produced by the full support through
$(x,y,z)^\top>0$. The first two equations of
$(A-t\Id_3)(x,y,z)^\top=0$ give
\begin{equation}\label{eq:forbidden-relation}
    (\theta_2-t)x=(\theta_1-t)y.
\end{equation}
If $\theta_1\neq\theta_2$, the two factors in this equality must have
the same sign because $x,y>0$. Thus no point of the closed interval
with endpoints $\theta_1$ and $\theta_2$ can be produced by the full
support.

A direct calculation gives
\begin{equation}\label{eq:endpoint-values}
    \chi_A(\theta_1)=(\theta_1-\theta_2)G,
    \qquad
    \chi_A(\theta_2)=(\theta_2-\theta_1)F.
\end{equation}
By \eqref{eq:F-nonnegative} and \eqref{eq:G-nonnegative}, these two
values have opposite signs or one of them is zero. Hence $\chi_A$ has
a root in the forbidden interval. Since $\chi_A$ has degree $3$, at
most two other distinct roots can be produced by the full support.

Now suppose that $\theta_1=\theta_2=: \theta$. The first two rows of
$A-\theta\Id_3$ are equal, so $\theta$ is an eigenvalue. A positive
kernel vector $(x,y,z)^\top$ would satisfy
$$
    z=rx+py,
    \qquad
    Fx+Gy=0.
$$
Since $F,G\geq0$ and $x,y>0$, this is impossible unless $F=G=0$. If
$(F,G)\neq(0,0)$, the value $\theta$ is not produced by the full
support, so at most the other two eigenvalues can be produced. If
$F=G=0$, the third row of $A-\theta\Id_3$ is $\theta$ times the first
row. Hence, $\dim\ker(A-\theta\Id_3)\geq2$. The algebraic multiplicity
of $\theta$ is then at least two, and $A$ has at most two distinct
eigenvalues. In both cases, $|\Pi_{\Ind_3}(A)|\leq2$.
\end{proof}

\section{The order-three capacity}
\label{sec:capacity}

We first give a matrix with $9$ distinct Pareto eigenvalues. We then
use the preceding lemmas to prove the upper bound.

\begin{lemma}\label{lem:nine-value-example}
The matrix
\begin{equation}\label{eq:c3-witness}
    A_0=
    \begin{pmatrix}
        5&-1&6\\
        6&0&6\\
        4&-1&7
    \end{pmatrix}
\end{equation}
has at least $9$ distinct Pareto eigenvalues.
\end{lemma}

\begin{proof}
Set
$$
    \eta_\pm:=\frac{11\pm\sqrt{73}}{2}.
$$
The table gives $9$ productions. The last column contains the outside
slack $A_{J^{\mathrm c},J}u$.
$$
\renewcommand{\arraystretch}{1.18}
\begin{array}{c|c|c|c}
J&\lambda&u&A_{J^{\mathrm c},J}u\\ \hline
\{1\}&5&1&(6,4)^\top\\
\{3\}&7&1&(6,6)^\top\\
\{1,2\}&2&(1,3)^\top&1\\
\{1,2\}&3&(1,2)^\top&2\\
\{2,3\}&1&(6,1)^\top&0\\
\{2,3\}&6&(1,1)^\top&5\\
\{1,3\}&11&(1,1)^\top&12\\
\{1,2,3\}&\eta_\pm&(1,11-\eta_\pm,1)^\top&\text{-}
\end{array}
$$
For the last line,
$$
    A_0
    \begin{pmatrix}1\\11-t\\1\end{pmatrix}
    =t\begin{pmatrix}1\\11-t\\1\end{pmatrix}
    \quad\Longleftrightarrow\quad
    t^2-11t+12=0.
$$
All the displayed vectors are positive and all outside slacks are
nonnegative. Since $8<\sqrt{73}<9$,
$$
    1<\eta_-<\frac32<2<3<5<6<7
    <\frac{19}{2}<\eta_+<10<11.
$$
Thus the $9$ values are distinct.
\end{proof}

The bound $N(A)\leq9$ also follows from
\cite[Corollary~1]{BaillonSeeger2021} together with the elementary
inequality $P_1(A)+P_3(A)\leq5$. At full capacity, this bound implies
that each Pareto eigenvalue is produced by a unique support. We give
an independent proof using the support restrictions of
Section~\ref{sec:structure}.

\begin{theorem}\label{thm:c3}
For every $A\in\R^{3\times3}$,
$$
    |\Pi(A)|\leq N(A)\leq9.
$$
The Pareto capacity in order $n=3$ is
$$
    c_3=9.
$$
If $A$ is full-capacity, each of its Pareto eigenvalues is
produced by exactly one support.
\end{theorem}

\begin{proof}
By \eqref{eq:support-union}, every value in $\Pi(A)$ has at
least one producing support, so $|\Pi(A)|\leq N(A)$. Write
$P_k=P_k(A)$. Corollary~\ref{cor:pair-bound} gives $P_2\leq5$, and
$P_3\leq3$ because every value produced by the full support is an
eigenvalue of $A$. A singleton $\{i\}$ can produce only $a_{ii}$.
Thus $P_1$ is the number of producing singleton supports.

If $P_1\leq1$, then $N(A)=P_1+P_2+P_3\leq1+5+3=9$.

If $P_1=3$, the singleton conditions in
\eqref{eq:support-criterion} make every off-diagonal entry of $A$
nonnegative. Lemma~\ref{lem:rich-pair} therefore rules out every
rich pair, so $P_2\leq3$. Lemma~\ref{lem:all-singletons} gives
$P_3\leq1$, and hence $N(A)\leq3+3+1=7$.

Suppose finally that $P_1=2$. Relabel the indices so that
$\{1\}$ and $\{2\}$ produce. Then $a_{21},a_{12}\geq0$, so
Lemma~\ref{lem:rich-pair} gives $p_{\{1,2\}}(A)\leq1$.
If $P_2\leq4$, then $N(A)\leq2+4+3=9$. If $P_2=5$, each of
$\{1,3\}$ and $\{2,3\}$ must produce two values, since a pair
produces at most two. Lemma~\ref{lem:interval-obstruction} now gives
$P_3\leq2$, so $N(A)\leq2+5+2=9$.

This proves the upper bound for every $A$. The matrix $A_0$ in
Lemma~\ref{lem:nine-value-example} attains it, so $c_3=9$.
If $|\Pi(A)|=9$, then $N(A)=9$ as well. By \eqref{eq:support-union},
the union of the seven finite sets $\Pi_J(A)$ then has cardinality
equal to the sum of their cardinalities. These sets are therefore
pairwise disjoint, so every Pareto eigenvalue has a unique producing
support.
\end{proof}

A matrix of order $3$ is therefore full-capacity exactly when it has
$9$ distinct Pareto eigenvalues.

\section{Classification of intrinsic support profiles}
\label{sec:classification}
For a full-capacity matrix $A\in\R^{3\times3}$, Theorem~\ref{thm:c3}
gives a unique producing support $\tau_A(\lambda)$ for each $\lambda\in\Pi(A)$.
Hence $m_J(A,\tau_A)=p_J(A)$ for every $J\in\Supp_3$ and
$S_k(A,\tau_A)=P_k(A)$ for $k=1,2,3$. Write $p_i=p_{\{i\}}(A)$,
$p_{ij}=p_{\{i,j\}}(A)$, $p_{123}=p_{\Ind_3}(A)$, and $P_k=P_k(A)$.
The aggregate support profile is $(P_1,P_2,P_3)$ and the detailed
support profile is $(p_1,p_2,p_3;p_{12},p_{13},p_{23};p_{123})$.

\begin{theorem}
\label{thm:classification}
Let $A\in\R^{3\times3}$ be full-capacity.
Up to a simultaneous permutation of the indices, the aggregate and
detailed support profiles of $A$ are given by one row of the following table.
\begin{equation}\label{eq:classified-profiles}
\begin{array}{c|c}
\toprule
\text{aggregate profile}&\text{detailed profile}\\
\midrule
(1,5,3)&(1,0,0;1,2,2;3)\\
(2,4,3)&(1,1,0;1,2,1;3)\\
(2,5,2)&(1,1,0;1,2,2;2)\\
\bottomrule
\end{array}
\end{equation}
\end{theorem}

\begin{proof}
By Theorem~\ref{thm:c3}, $P_1+P_2+P_3=N(A)=9$.
Corollary~\ref{cor:pair-bound} gives $P_2\leq5$, and $P_3\leq3$,
so $P_1\geq1$. If $P_1=3$, the proof of Theorem~\ref{thm:c3}
gives $N(A)\leq7$. Thus $P_1$ is either $1$ or $2$.
For $P_1=1$, the bounds on $P_2$ and $P_3$ force $(P_2,P_3)=(5,3)$.
For $P_1=2$, we have $P_2+P_3=7$, giving $(P_2,P_3)=(4,3)$ or $(5,2)$.

Suppose first that $P_1=1$, and put the producing singleton at
index $1$. Since $P_2=5$ and each pair produces at most two values,
the pair counts are a permutation of $(1,2,2)$. If both rich pairs
contained index $1$, Lemma~\ref{lem:active-center} would give $p_{23}=0$,
whereas the third pair must produce one value. Thus $\{2,3\}$ is rich.
After interchanging $2$ and $3$ if necessary, the other rich pair is
$\{1,3\}$. This gives the first row of \eqref{eq:classified-profiles}.

Now let $P_1=2$, and take $\{1\}$ and $\{2\}$ as the producing
singletons. Then $a_{12},a_{21}\geq0$, so Lemma~\ref{lem:rich-pair}
gives $p_{12}\leq1$. If $(P_2,P_3)=(4,3)$ and $p_{12}=0$, then
$p_{13}=p_{23}=2$. Lemma~\ref{lem:interval-obstruction} would give
$P_3\leq2$, a contradiction. Hence $p_{12}=1$, and the other two
counts sum to three. Interchanging $1$ and $2$ if necessary gives
$(p_{12},p_{13},p_{23})=(1,2,1)$, as in the second row.
If $(P_2,P_3)=(5,2)$, the three pair counts must attain their
respective upper bounds $1,2,2$. This gives the third row.
\end{proof}

\section{Regular realizations}
\label{sec:realizations}

We shall use the next lemma to check regular production. In the tables,
we write $12$, $13$, $23$, and $123$ for the corresponding supports.
Throughout this section,
$$
    \chi_J(t):=\det(t\Id_{|J|}-A_J).
$$
For a square matrix $M$, we denote by $\adj(M)$ its adjugate matrix.

The next lemma gives an explicit certificate for regular production
based on the adjugate matrix.
\begin{lemma}\label{lem:adjugate-certificate}
Let $J\in\Supp_n$, and let $\lambda$ be a simple real root of
$\chi_J$. Suppose that, for a coordinate vector $e_k$ and a sign
$\sigma\in\{-1,1\}$,
$$
    v(t):=\adj(t\Id_{|J|}-A_J)e_k
$$
satisfies
$$
    \sigma v(\lambda)>0,
    \qquad
    \sigma A_{J^{\mathrm c},J}v(\lambda)>0.
$$
Then $J$ regularly produces $\lambda$.
\end{lemma}

\begin{proof}
The adjugate identity gives
$$
    (\lambda\Id_{|J|}-A_J)v(\lambda)=0.
$$
Since $\lambda$ is simple, the matrix
$\lambda\Id_{|J|}-A_J$ has rank $|J|-1$. A nonzero column of its
adjugate is therefore an eigenvector. The two inequalities give a
positive eigenvector and positive outside slacks.
\end{proof}

\subsection{The profile \texorpdfstring{$(1,5,3)$}{(1,5,3)}}

\begin{proposition}\label{prop:A153}
The matrix
\begin{equation}\label{eq:A153}
    A_{153}:=
    \begin{pmatrix}
        16&8&-11\\
        9&-16&5\\
        14&-1&-11
    \end{pmatrix}
\end{equation}
has $9$ distinct regular Pareto eigenvalues. Its detailed support profile is
$$
    (1,0,0;1,2,2;3).
$$
\end{proposition}

\begin{proof}
The support $\{1\}$ regularly produces $16$. Its outside slacks are
$9$ and $14$. The supports $\{2\}$ and $\{3\}$ do not produce their
diagonal values, because $a_{32}=-1$ and $a_{13}=-11$.

The table lists the characteristic polynomials and intervals used for
the other eight values. The last column gives the values of the
polynomial at the endpoints.
$$
\renewcommand{\arraystretch}{1.15}
\begin{array}{c|c|c|c}
\toprule
J&\chi_J(t)&I=(a,b)&(\chi_J(a),\chi_J(b))\\
\midrule
12&t^2-328&(18,19)&(-4,33)\\
13&t^2-5t-22&(-3,-2)&(2,-8)\\
13&t^2-5t-22&(7,8)&(-8,2)\\
23&t^2+27t+181&(-15,-14)&(1,-1)\\
23&t^2+27t+181&(-25/2,-99/8)&(-1/4,1/64)\\
123&t^3+11t^2-169t-1883&(-49/4,-61/5)&(-21/64,24/125)\\
123&t^3+11t^2-169t-1883&(-119/10,-59/5)&(651/1000,-24/125)\\
123&t^3+11t^2-169t-1883&(13,14)&(-24,651)\\
\bottomrule
\end{array}
$$
The discriminants of the $3$ quadratic polynomials are $1312$,
$113$, and $5$. Their roots are therefore simple. The $3$ cubic
intervals are disjoint and contain sign changes. Hence they contain
$3$ distinct real roots. Since the polynomial has degree $3$,
these are all its roots and they are simple.

For the pair supports, apply
Lemma~\ref{lem:adjugate-certificate} with
$$
\begin{array}{c|c|c}
J&v_J(t)&A_{J^{\mathrm c},J}v_J(t)\\ \hline
12&(t+16,9)^\top&14t+215\\
13&(t+11,14)^\top&9t+169\\
23&(5,t+16)^\top&-11t-136.
\end{array}
$$
The vectors and the slacks are positive at the roots in the listed
intervals. For the second interval for $23$, for instance,
$t<-99/8$ gives $-11t-136>1/8$. The other root of $t^2-328$ lies in
$(-19,-18)$. On that interval, $(t+16,9)^\top$ has components of
opposite signs. Thus $12$ produces only its positive root.

For the full support, the first column of the adjugate is
\begin{equation}\label{eq:A153-vector}
    v(t)=
    \begin{pmatrix}
        t^2+27t+181\\
        9t+169\\
        14t+215
    \end{pmatrix}.
\end{equation}
It is positive on the $3$ cubic intervals. Indeed,
$h(t):=t^2+27t+181$ is increasing for $t>-27/2$, and
$h(-49/4)=5/16>0$. The two affine components are also positive at
$t=-49/4$, and therefore on all $3$ intervals.

The eight intervals and the singleton value $16$ are pairwise
disjoint. They give $9$ distinct regular Pareto eigenvalues.
By Theorem~\ref{thm:c3}, these exhaust $\Pi(A_{153})$, and each is
produced only on the support indicated above. Counting the values
by support gives $(1,0,0;1,2,2;3)$, with aggregate profile $(1,5,3)$.
\end{proof}
We shall also need the spectra of all principal submatrices of
$A_{153}$. Table~\ref{tab:A153-principal-spectra} gives pairwise
disjoint sets containing these spectra, together with numerical
approximations of the eigenvalues. The values in bold are the Pareto
eigenvalues produced by the corresponding support. It follows that
two distinct principal submatrices of $A_{153}$ have no common
eigenvalue.

\begin{table}[htbp]
\centering
\small
\renewcommand{\arraystretch}{1.16}
\begin{tabular}{
@{}c
>{\raggedright\arraybackslash}p{0.48\textwidth}
>{\raggedright\arraybackslash}p{0.34\textwidth}
@{}}
\toprule
$J$
&
Location of $\operatorname{spec}((A_{153})_J)$
&
Numerical approximations
\\
\midrule

$1$
&
$\{16\}$
&
$\boldsymbol{16}$
\\

$2$
&
$\{-16\}$
&
$-16$
\\

$3$
&
$\{-11\}$
&
$-11$
\\

$12$
&
One eigenvalue in each of $(-19,-18)$ and $(18,19)$
&
$-18.110770,\ \boldsymbol{18.110770}$
\\

$13$
&
One eigenvalue in each of $(-3,-2)$ and $(7,8)$
&
$\boldsymbol{-2.815073},\ \boldsymbol{7.815073}$
\\

$23$
&
One eigenvalue in each of $(-15,-14)$ and
$(-25/2,-99/8)$
&
$\boldsymbol{-14.618034},\
 \boldsymbol{-12.381966}$
\\

$123$
&
One eigenvalue in each of $(-49/4,-61/5)$,
$(-119/10,-59/5)$, and $(13,14)$
&
$\boldsymbol{-12.219936},\ \boldsymbol{-11.818408},$\newline
$\boldsymbol{13.038344}$
\\

\bottomrule
\end{tabular}
\caption{Locations and numerical approximations of the principal
spectra of $A_{153}$. The values in bold are the Pareto eigenvalues
produced by the indicated support.}
\label{tab:A153-principal-spectra}
\end{table}
\subsection{The profile \texorpdfstring{$(2,4,3)$}{(2,4,3)}}

\begin{proposition}\label{prop:A243}
The matrix
\begin{equation}\label{eq:A243}
    A_{243}:=
    \begin{pmatrix}
        0&10&-4\\
        10&9&-6\\
        4&2&-9
    \end{pmatrix}
\end{equation}
has $9$ distinct regular Pareto eigenvalues. Its detailed support profile is
$$
    (1,1,0;1,2,1;3).
$$
\end{proposition}

\begin{proof}
The supports $\{1\}$ and $\{2\}$ regularly produce $0$ and $9$,
respectively. Their outside slacks are strictly positive. The support
$\{3\}$ does not produce $-9$, because
$$
    a_{13}<0
    \qquad\text{and}\qquad
    a_{23}<0.
$$

Table~\ref{tab:A243-principal-spectra} gives the characteristic
polynomials, root locations, and numerical approximations for all
principal submatrices. The values in bold are the Pareto eigenvalues
produced by the corresponding support.

\begin{table}[htbp]
\centering
\small
\renewcommand{\arraystretch}{1.15}
\begin{tabular}{
@{}c
>{\raggedright\arraybackslash}p{0.19\textwidth}
>{\raggedright\arraybackslash}p{0.42\textwidth}
>{\raggedleft\arraybackslash}p{0.22\textwidth}
@{}}
\toprule
$J$
&
$\chi_J(t)$
&
Interval and endpoint values
&
Approximate eigenvalue
\\
\midrule

$1$
&
$t$
&
$\{0\}$
&
$\boldsymbol{0}$
\\

$2$
&
$t-9$
&
$\{9\}$
&
$\boldsymbol{9}$
\\

$3$
&
$t+9$
&
$\{-9\}$
&
$-9$
\\

$12$
&
$t^2-9t-100$
&
$(-13/2,-32/5)$, with values
$(3/4,-36/25)$
&
$-6.465856$
\\

$12$
&
$t^2-9t-100$
&
$(15,16)$, with values $(-10,12)$
&
$\boldsymbol{15.465856}$
\\

$13$
&
$t^2+9t+16$
&
$(-33/5,-13/2)$, with values
$(4/25,-1/4)$
&
$\boldsymbol{-6.561553}$
\\

$13$
&
$t^2+9t+16$
&
$(-5/2,-12/5)$, with values
$(-1/4,4/25)$
&
$\boldsymbol{-2.438447}$
\\

$23$
&
$t^2-69$
&
$(-17/2,-8)$, with values $(13/4,-5)$
&
$-8.306624$
\\

$23$
&
$t^2-69$
&
$(8,9)$, with values $(-5,12)$
&
$\boldsymbol{8.306624}$
\\

$123$
&
$t^3-153t-724$
&
$(-8,-15/2)$, with values $(-12,13/8)$
&
$\boldsymbol{-7.591235}$
\\

$123$
&
$t^3-153t-724$
&
$(-67/10,-33/5)$, with values
$(337/1000,-212/125)$
&
$\boldsymbol{-6.681971}$
\\

$123$
&
$t^3-153t-724$
&
$(14,143/10)$, with values
$(-122,12307/1000)$
&
$\boldsymbol{14.273206}$
\\

\bottomrule
\end{tabular}
\caption{Locations and numerical approximations of the principal
spectra of $A_{243}$. The values in bold are the Pareto eigenvalues
produced by the indicated support.}
\label{tab:A243-principal-spectra}
\end{table}

The discriminants of the $3$ quadratic polynomials are $481$, $17$,
and $276$. Hence all their roots are simple. The $3$ intervals for
$\chi_{123}$ contain sign changes. They are disjoint, so
$\chi_{123}$ has $3$ distinct real roots. Since it has degree $3$,
these are all its roots and they are simple.

For the pair supports, use the following adjugate certificates:
$$
\renewcommand{\arraystretch}{1.15}
\begin{array}{c|c|c}
J&v_J(t)&A_{J^{\mathrm c},J}v_J(t)\\
\hline
12&(10,t)^\top&40+2t\\
13&(t+9,4)^\top&10t+66\\
23&(t+9,2)^\top&10t+82.
\end{array}
$$
The vectors and the outside slacks are strictly positive at the pair
roots shown in bold in Table~\ref{tab:A243-principal-spectra}.
Therefore these roots are regularly produced.

The negative root of $t^2-9t-100$ has an eigenvector with components
of opposite signs. Hence it is not produced by the support $\{1,2\}$.

The negative root of $t^2-69$ has a positive eigenvector
$$
    (9-\sqrt{69},2)^\top.
$$
However, its outside slack is
$$
    82-10\sqrt{69}<0,
$$
because $\sqrt{69}>41/5$. Hence this root is not produced by the
support $\{2,3\}$. The pair counts are therefore
$$
    p_{12}=1,
    \qquad
    p_{13}=2,
    \qquad
    p_{23}=1.
$$

For the full support, column $1$ of the adjugate is
\begin{equation}\label{eq:A243-vector}
    v(t)=
    \begin{pmatrix}
        t^2-69\\
        10t+66\\
        4t-16
    \end{pmatrix}.
\end{equation}
All $3$ components are negative on the $2$ negative cubic intervals
in Table~\ref{tab:A243-principal-spectra}. Multiplying $v(t)$ by $-1$
therefore gives a positive eigenvector. All $3$ components are
positive on the positive cubic interval. Thus the full support
regularly produces its $3$ eigenvalues.

The locations in Table~\ref{tab:A243-principal-spectra} are pairwise
disjoint. Hence distinct principal submatrices have disjoint spectra.
The table contains all eigenvalues of all principal submatrices. The
$3$ nonbold values are not produced, while the $9$ bold values are
regularly produced. Therefore $A_{243}$ has exactly $9$ distinct
regular Pareto eigenvalues, with detailed profile
$$
    (1,1,0;1,2,1;3).
$$
\end{proof}
\subsection{The profile \texorpdfstring{$(2,5,2)$}{(2,5,2)}}

The family considered below is the order-$3$ case of the family in
\cite[Eq.~(2.5)]{SeegerVicente2011}. In that reference, the parameter
is assumed to satisfy $s>1+\sqrt{2}$ and to be transcendental. In that
proof, the transcendence assumption is used to exclude coincidences
between values produced by supports containing index $1$ and supports
not containing index $1$. Baillon and Seeger~\cite{BaillonSeeger2020} observed that all Pareto eigenvalues of this
family are regular for transcendental $s>1+\sqrt{2}$.
Proposition~\ref{prop:A252-family} proves that, in order $3$,
$|\Pi(A_{252}(s))|=9$ for every real $s>1+\sqrt{2}$, thus
settling the order-$3$ case of the conjecture in
\cite[Remark~2.3]{SeegerVicente2011}. It also establishes
regularity of all $9$ values throughout this interval and
determines their detailed support profile.

For $s>1+\sqrt{2}$, define
\begin{equation}\label{eq:A252-family}
    A_{252}(s):=
    \begin{pmatrix}
        s^4&s^5&-s^3\\
        s^5&s^6&-s^4\\
        s^3&s^4&s^2
    \end{pmatrix}.
\end{equation}
\begin{proposition}
\label{prop:A252-family}
For every real $s>1+\sqrt{2}$, the matrix $A_{252}(s)$ has $9$
distinct regular Pareto eigenvalues. Its detailed support profile is
$$
    (1,1,0;1,2,2;2).
$$
\end{proposition}

\begin{proof}
Permuting the indices in the order $(3,1,2)$ transforms $A_{252}(s)$
into
\begin{equation}\label{eq:Atilde-family}
    \widetilde A(s)=
    \begin{pmatrix}
        s^2&s^3&s^4\\
        -s^3&s^4&s^5\\
        -s^4&s^5&s^6
    \end{pmatrix}.
\end{equation}
This is the matrix in \cite[Eq.~(2.5)]{SeegerVicente2011} for $n=3$.
By Proposition~\ref{prop:invariance}, it is enough to study
$\widetilde A(s)$. Put $
    q:=s^2.$
Then $
    q>3+2\sqrt{2}>5.$

The support $\{1\}$ does not produce its diagonal value $q$, because
its outside slack is
$$
    \begin{pmatrix}
        -s^3\\
        -s^4
    \end{pmatrix}.
$$
The supports $\{2\}$ and $\{3\}$ regularly produce $q^2$ and $q^3$.
Their outside slacks are, respectively,
$$
    \begin{pmatrix}
        s^3\\
        s^5
    \end{pmatrix}
    \qquad\text{and}\qquad
    \begin{pmatrix}
        s^4\\
        s^5
    \end{pmatrix},
$$
and are strictly positive.

For the support $\{1,2\}$, write a positive eigenvector as
$(1,z)^\top$. The ratio $z$ satisfies
\begin{equation}\label{eq:A252-ratio-12}
    f_{12}(z):=
    z^2-\left(s-\frac1s\right)z+1=0.
\end{equation}
Since
$$
    s-\frac1s>2,
$$
the discriminant is positive. The sum and product of the roots are
positive. Hence both roots are positive. Let
$$
    0<z_-<z_+
$$
be these roots. Since $f_{12}(1)<0$, we have
$$
    z_-<1<z_+.
$$
Moreover,
$$
    f_{12}\left(\frac1s\right)=\frac{2}{s^2}>0.
$$
Since $1/s<1$, it follows that
$$
    \frac1s<z_-<1<z_+.
$$
The outside slack is
$$
    s^4(sz-1),
$$
which is strictly positive at both roots. The associated eigenvalue is
$$
    \lambda=q+s^3z.
$$
The roots therefore give $2$ distinct simple eigenvalues. Thus
$\{1,2\}$ regularly produces $2$ values.

For the support $\{1,3\}$, the ratio equation is
\begin{equation}\label{eq:A252-ratio-13}
    f_{13}(z):=
    z^2-\left(s^2-\frac1{s^2}\right)z+1=0.
\end{equation}
Since
$$
    q-\frac1q>2,
$$
the discriminant is positive, and both roots are positive. Let
$$
    0<w_-<w_+
$$
be the roots. Since $f_{13}(1)<0$ and
$$
    f_{13}\left(\frac1{s^2}\right)=\frac{2}{s^4}>0,
$$
we obtain
$$
    \frac1{s^2}<w_-<1<w_+.
$$
At either root $w$, the outside slack is
$$
    s^3(s^2w-1)>0.
$$
The associated eigenvalue is
$$
    \lambda=q+q^2w.
$$
Hence $\{1,3\}$ regularly produces $2$ values.

The principal submatrix on $\{2,3\}$ is entrywise positive and has
rank $1$. The vector $(1,s)^\top$ gives the simple nonzero eigenvalue
$$
    q^2+q^3.
$$
The outside slack is
$$
    s^3+s^5>0.
$$
The zero eigenvalue has eigenvectors proportional to $(-s,1)^\top$,
so it is not produced. Thus $\{2,3\}$ regularly produces exactly $1$
value.

For the full support,
$$
    \det(t\Id_3-\widetilde A(s))=t\,\phi(t),
$$
where
$$
    \phi(t)=t^2-q(1+q+q^2)t+2q^3(1+q).
$$
Its discriminant is
$$
    q^2\left((1+q+q^2)^2-8q(1+q)\right).
$$
Since $q>5$,
$$
    q^3>25q>8(q+1).
$$
Therefore,
$$
    (1+q+q^2)^2>q^4>8q(1+q),
$$
and the discriminant is positive. Thus $\phi$ has $2$ distinct real
roots. Their sum and product are positive, so both roots are positive.

Let
$$
    0<\lambda_-<\lambda_+
$$
be the roots of $\phi$. We have
$$
    \phi(2q)=2q^2>0
$$
and
$$
    \phi'(2q)=q(3-q-q^2)<0.
$$
Hence $2q$ lies to the left of the vertex of $\phi$.
Since the leading coefficient of $\phi$ is equal to $1$, the polynomial
is negative between its two roots. It follows that
$$
    2q<\lambda_-<\lambda_+.
$$

For either root $\lambda$, direct substitution gives the eigenvector
\begin{equation}\label{eq:A252-full-vector}
    u_\lambda=
    \begin{pmatrix}
        \dfrac{\lambda}{s^2(\lambda-2s^2)}\\[3pt]
        \dfrac1s\\[2pt]
        1
    \end{pmatrix}.
\end{equation}
This vector is positive because $\lambda>2s^2$. The roots of $\phi$
are simple and nonzero. The zero eigenvalue is simple because
$$
    \phi(0)=2q^3(1+q)>0.
$$
The eigenspace for $0$ is spanned by $(0,-s,1)^\top$ and therefore
contains no strictly positive vector. Hence the full support regularly produces the $2$
nonzero roots of $\phi$.

We now prove that the $9$ produced values are distinct. The $2$ values
from $12$, the $2$ values from $13$, and the $2$ nonzero values from
$123$ are the roots of $g_{12}$, $g_{13}$, and $\phi$, respectively, where
$$
    g_{12}(t)=t^2-(q+q^2)t+2q^3,\qquad
    g_{13}(t)=t^2-(q+q^3)t+2q^4.
$$
Their pairwise resultants are
$$
    \Res_t(g_{12},g_{13})=2q^6(q-1)^2,\;\;
    \Res_t(g_{12},\phi)=2q^8,\;\;
    \Res_t(g_{13},\phi)=2q^6,
$$
where for polynomials $f$ and $g$, we denote by $\Res_t(f,g)$ their resultant with respect to $t$.
These resultants are nonzero because $q>5$. Hence the $3$ polynomials
have pairwise disjoint root sets.

Their values at the other $3$ produced values are
$$
\renewcommand{\arraystretch}{1.15}
\begin{array}{c|ccc}
& q^2&q^3&q^2+q^3\\
\hline
g_{12}
& q^3
& q^3(q^3-q^2-q+2)
& q^3(q^3+q^2-q+1)
\\
g_{13}
& -q^3(q^2-3q+1)
& q^4
& q^3(q^2+2q-1)
\\
\phi
& -q^3(q^2-2q-1)
& -q^3(q-2)(q+1)
& q^3(q+1).
\end{array}
$$
All these quantities are nonzero. Indeed, $q>5$ gives
$$
\begin{gathered}
    q^3-q^2-q+2>0,
    \qquad
    q^3+q^2-q+1>0,\\
    q^2-3q+1>0,
    \qquad
    q^2+2q-1>0,
    \qquad
    q^2-2q-1>0,
\end{gathered}
$$
and $
    q-2>0.$
Also,
$$
    q^2<q^3<q^2+q^3.
$$
The roots within each polynomial are distinct, the $3$ root sets are
pairwise disjoint, and none of these roots is equal to one of the
remaining $3$ produced values. Therefore the $9$ produced values are
distinct.

By Theorem~\ref{thm:c3}, these nine values exhaust $\Pi(\widetilde A(s))$,
and each has a unique producing support. The detailed support profile of
$\widetilde A(s)$ is
$$
    (0,1,1;2,2,1;2).
$$
After returning to the original order of the indices, the detailed
support profile of $A_{252}(s)$ is
$$
    (1,1,0;1,2,2;2).
$$
\end{proof}

For $s=3$, positive scaling gives the smaller integer matrix
\begin{equation}\label{eq:A252-integer}
    A_{252}:=\frac19A_{252}(3)=
    \begin{pmatrix}
        9&27&-3\\
        27&81&-9\\
        3&9&1
    \end{pmatrix}.
\end{equation}
Under the same permutation $(3,1,2)$,
$$
    \frac19\widetilde A(3)=
    \begin{pmatrix}
        1&3&9\\
        -3&9&27\\
        -9&27&81
    \end{pmatrix},
$$
which is the order-$3$ example in
\cite[Example~1.2]{SeegerVicente2011}.

Propositions~\ref{prop:A153}, \ref{prop:A243}, and
\ref{prop:A252-family} give a regular realization of every profile in
Theorem~\ref{thm:classification}.

\section{Open realizations and intrinsic profiles}
\label{sec:stability}

Let $\Pi^{\rm reg}(A)$ denote the set of regular Pareto eigenvalues of
$A$. Following \cite{BaillonSeeger2020}, the regular Pareto capacity is
defined by
$$
    c_n^{\rm reg}
    :=
    \max_{A\in\R^{n\times n}}
    |\Pi^{\rm reg}(A)|.
$$
It is known that
$$
    c_3^{\rm reg}=c_3=9.
$$
Hence each matrix constructed in Section~\ref{sec:realizations}
maximizes both the number of Pareto eigenvalues and the number of
regular Pareto eigenvalues. Let $\mathcal F$ denote the set of full-capacity matrices of
order three and $\mathcal R$ its subset with 9 regular values. Thus
$$
  \mathcal F=\{A\in\R^{3\times3}:|\Pi(A)|=9\},\qquad
  \mathcal R=\{A\in\mathcal F:|\Pi^{\rm reg}(A)|=9\}.
$$
By Theorem~\ref{thm:c3}, the profile of every $A\in\mathcal F$ is
intrinsic. On $\mathcal R$ we can also follow the individual values
on their producing supports under perturbation.

Kielstra~\cite[Definition~3.6.6 and Theorem~3.6.7]{Kielstra2023}
defines the Pareto profile as the vector $(p_J(A))_{J\in\Supp_n}$,
where $p_J(A)$ counts the distinct values produced by $J$. He proves
that this vector is constant on each connected component of the set of
$1$-principally simple matrices. In order three,
Proposition~\ref{cor:open-realizations} shows that the same vector is
locally constant at every matrix with 9 distinct regular Pareto
eigenvalues.

\subsection{Local constancy on the regular part}
We use the following stability result of Baillon and Seeger
\cite[Proposition~3]{BaillonSeeger2020}.
Suppose that a support $J$ regularly produces $\lambda$ for a matrix
$A\in\R^{3\times3}$. Then there exist an open neighborhood $\mathcal V$
of $A$ in $\R^{3\times3}$ and a continuous function
$\psi:\mathcal V\to\R$ with $\psi(A)=\lambda$ such that, for every
$B\in\mathcal V$, the same support $J$ regularly produces $\psi(B)$.

\begin{proposition}
\label{cor:open-realizations}
The set $\mathcal R$ is open in $\R^{3\times3}$. For every
$A\in\mathcal R$ there is an open neighborhood $\mathcal V$ of $A$
contained in $\mathcal R$ such that $p_J(B)=p_J(A)$ for every
$B\in\mathcal V$ and every $J\in\Supp_3$.
\end{proposition}
\begin{proof}
Write $\Pi(A)=\{\lambda_1,\ldots,\lambda_9\}$, and let $J_i$ be
the unique support producing $\lambda_i$. Since $\lambda_i$ is regular,
$J_i$ regularly produces it. Choose pairwise disjoint open intervals
$I_i$ containing the respective $\lambda_i$. By supportwise stability
and continuity, for each $i$ there are an open neighborhood
$\mathcal V_i$ of $A$ and a continuous function
$\psi_i:\mathcal V_i\to I_i$ such that $\psi_i(A)=\lambda_i$ and $J_i$
regularly produces $\psi_i(B)$ for every $B\in\mathcal V_i$.

On $\mathcal V=\bigcap_{i=1}^9\mathcal V_i$ the nine values are
distinct. By Theorem~\ref{thm:c3} they exhaust $\Pi(B)$, and each
$J_i$ is the unique support producing $\psi_i(B)$. Thus $B\in\mathcal R$
and $p_J(B)=|\{i:J_i=J\}|=p_J(A)$ for every support $J$.
\end{proof}

\begin{remark}
For fixed indices, the detailed profile is therefore constant on
each connected component of $\mathcal R$, since each level set of the
profile and its complement are open in $\mathcal R$.
\end{remark}

\subsection{A Zariski-open separation condition}

Baillon and Seeger \cite{BaillonSeeger2021} call a matrix \emph{spectrally separable} if every
Pareto eigenvalue has a unique producing support. They also observe
that spectrally separable matrices form a dense subset of the space of
matrices \cite[Section~1.1]{BaillonSeeger2021}. The next proposition
does not claim this density result as new. It gives an explicit
Zariski-open sufficient condition for spectral separability.

The criterion can be tested on any matrix of order three.
It separates the complete complex spectra of distinct
principal submatrices.

\begin{proposition}
\label{prop:generic-uniqueness}
There is a nonempty Zariski-open set $
    \mathcal U\subset\R^{3\times3}$
which is also Euclidean-open and dense, such that
$$
    \operatorname{spec}(A_J)
    \cap
    \operatorname{spec}(A_K)
    =
    \varnothing
$$
for every $A\in\mathcal U$ and every pair of distinct supports
$J,K\in\Supp_3$. Consequently, every Pareto eigenvalue of a matrix in
$\mathcal U$ has a unique producing support.
\end{proposition}

\begin{proof}
For distinct supports $J,K\in\Supp_3$, define
$$
    R_{J,K}(A)
    :=
    \Res_t\!\left(
        \det(t\Id_{|J|}-A_J),
        \det(t\Id_{|K|}-A_K)
    \right).
$$
This is a polynomial in the $9$ entries of $A$. Let $R(A)$ be the
product of $R_{J,K}(A)$ over all unordered pairs of distinct supports
$J,K\in\Supp_3$.

The polynomial $R$ is not identically zero.
Table~\ref{tab:A153-principal-spectra} contains all eigenvalues of all
principal submatrices of $A_{153}$, and these $7$ spectra are pairwise
disjoint. Hence, $
    R(A_{153})\neq0.$

Set $
    \mathcal U
    :=
    \{A\in\R^{3\times3}:R(A)\neq0\}.$
This set is nonempty and Zariski-open. It is Euclidean-open by
continuity of $R$. It is dense because the zero set of a nonzero real
polynomial has empty Euclidean interior.

Let $A\in\mathcal U$. Since $
    R_{J,K}(A)\neq0
$
for every pair of distinct supports, the resultant criterion gives
$$
    \operatorname{spec}(A_J)
    \cap
    \operatorname{spec}(A_K)
    =
    \varnothing.
$$
A Pareto eigenvalue produced by $J$ is an eigenvalue of $A_J$.
Therefore the same Pareto eigenvalue cannot be produced by a distinct
support $K$.
\end{proof}

\begin{remark}
The set $\mathcal U$ is a concrete principal Zariski-open subset of
the set of spectrally separable matrices. Its defining condition is
stronger than spectral separability, since it separates all complex
principal eigenvalues, including those which are not Pareto
eigenvalues. This stronger condition is useful because it is expressed
by one nonvanishing polynomial.

For every $s>1+\sqrt{2}$, Proposition~\ref{prop:A252-family}
gives $A_{252}(s)\in\mathcal R$. Its principal submatrix on $\{1,2\}$
and its full matrix both have the ordinary eigenvalue $0$. As the
proof of that proposition shows, neither support produces $0$. Thus
$A_{252}(s)\notin\mathcal U$, although all 9 Pareto eigenvalues
are regular and the detailed profile is locally constant.
\end{remark}

\subsection{Partition by aggregate profile}
For $A\in\mathcal F$, write $\mathbf P(A)=(P_1(A),P_2(A),P_3(A))$.
For each $\boldsymbol\alpha$ in
$\mathcal A=\{(1,5,3),(2,4,3),(2,5,2)\}$, set
$$
 \mathcal F_{\boldsymbol\alpha}
  =\{A\in\mathcal F:\mathbf P(A)=\boldsymbol\alpha\},\qquad
 \mathcal R_{\boldsymbol\alpha}=\mathcal R\cap\mathcal F_{\boldsymbol\alpha}.
$$
By Theorem~\ref{thm:classification}, the three
$\mathcal F_{\boldsymbol\alpha}$ partition $\mathcal F$.

\begin{theorem}
\label{cor:intrinsic-open}
The three sets $\mathcal R_{\boldsymbol\alpha}$,
$\boldsymbol\alpha\in\mathcal A$, form a partition of $\mathcal R$
into nonempty open subsets of $\R^{3\times3}$. With indices fixed,
the detailed profiles occurring in $\mathcal R_{\boldsymbol\alpha}$
are precisely the simultaneous permutations of the corresponding
representative in \eqref{eq:classified-profiles}.
\end{theorem}

\begin{proof}
Intersecting the partition of $\mathcal F$ with $\mathcal R$
gives the partition in the statement. By Proposition~\ref{cor:open-realizations},
each $A\in\mathcal R_{\boldsymbol\alpha}$ has an open neighborhood
on which all 9 values remain regular and each $p_J$ is constant.
Thus $\mathcal R_{\boldsymbol\alpha}$ is open. It contains the
corresponding regular realization from Section~\ref{sec:realizations}.
By Proposition~\ref{prop:invariance}, simultaneously permuting its
rows and columns preserves regularity and permutes its producing
supports. This realizes all permutations of the detailed profile,
and Theorem~\ref{thm:classification} exhausts the possibilities.
\end{proof}

Thus the three open sets in Theorem~\ref{cor:intrinsic-open} cover
exactly $\mathcal R$. A full-capacity matrix lies outside their union
if and only if it has at least one irregular Pareto eigenvalue.
Even then, each of its 9 values has a unique producing support by
Theorem~\ref{thm:c3}. For example, $A_0$ in \eqref{eq:c3-witness}
lies in $\mathcal F\setminus\mathcal R$ because its value $1$ has
zero outside slack on its unique producing support $\{2,3\}$
(Proposition~\ref{prop:profile-transition}).

\subsection{A change of profile at full capacity}
\begin{proposition}\label{prop:profile-transition}
Let $A(h)=A_0+hE_{13}$, where $A_0$ is given by
\eqref{eq:c3-witness} and $E_{13}$ has its only nonzero entry, equal
to $1$, in position $(1,3)$. For all sufficiently small $|h|$,
$A(h)$ has full capacity and
$$
 \mathbf P(A(h))=
 \begin{cases}
  (2,4,3),& h<0,\\
  (2,5,2),& h\geq0.
 \end{cases}
$$
For $h\ne0$ in this neighborhood all 9 values are regular.
At $h=0$ precisely the value $1$ is irregular. In particular,
$\mathcal R\subsetneq\mathcal F$.
\end{proposition}

\begin{proof}
At $h=0$ the support $\{2,3\}$ produces $1$ through
$(6,1)^\top$. Its outside slack is $-6+6=0$. This eigenvalue of the
pair is simple (the other eigenvalue is $6$), so every positive
eigenvector on the pair has zero outside slack. Theorem~\ref{thm:c3}
excludes any other producing support for $1$.

The other eight values in Lemma~\ref{lem:nine-value-example}
are regular. The outside slacks for both singletons and for the
four other pair productions are strictly positive. The three pair
blocks have respective eigenvalues $\{2,3\}$, $\{1,11\}$, and
$\{1,6\}$ in the order $12,13,23$. The characteristic polynomial of
$A_0$ is $(t-1)(t^2-11t+12)$, so the two values produced by the full
support are simple as well. By supportwise stability, the eight values
persist, on the same supports, in disjoint intervals separated from $1$.

The $\{2,3\}$ block is independent of $h$. At its eigenvalue
$1$, the same vector $(6,1)^\top$ now has outside slack $h$. Hence
this pair produces $1$ exactly when $h\geq0$, and does so regularly
when $h>0$. To examine the full support near $1$, write a candidate
eigenvector as $(x_1,x_2,1)^\top$. The last two rows of its
eigenvalue equation give
$$
 x_2=4x_1+7-t,\qquad
 x_1=\frac{(t-1)(t-6)}{4t-6}.
$$
The denominator is nonzero near $t=1$. The first row becomes
$(t-1)(x_1-1)=h$, which gives the characteristic equation
$$
 \det(t\Id_3-A(h))=(t-1)(t^2-11t+12)-h(4t-6)=0.
$$
The derivative of the left-hand side with respect to $t$ at
$(t,h)=(1,0)$ is $2$. The implicit function theorem gives a simple
eigenvalue branch $t(h)=1-h+O(h^2)$. Substitution yields
$x_1(h)=-5h/2+O(h^2)$ and $x_2(h)=6+O(h)$. For $h<0$ sufficiently
close to zero this branch has a strictly positive eigenvector and
is regularly produced by the full support. For $h>0$ the first and
third components have opposite signs. Since the eigenvalue is simple,
its eigenspace contains no strictly positive vector.

The eight persistent values contribute $(2,4,2)$ by support
size. The ninth lies on the full support for $h<0$ and on $\{2,3\}$
for $h\geq0$. Theorem~\ref{thm:c3} excludes any further production.
This proves the asserted profiles and places $A_0$ on the common
boundary, relative to $\mathcal F$, of $\mathcal F_{(2,4,3)}$ and
$\mathcal F_{(2,5,2)}$.
\end{proof}

\section{Conclusion}

For every real matrix of order three, the number of pairs
$(J,\lambda)$ for which $J$ produces $\lambda$ is at most nine.
At full capacity this implies that each of the nine distinct Pareto
eigenvalues has a unique producing support. The resulting aggregate
profiles are
$$
    (1,5,3),\qquad (2,4,3),\qquad (2,5,2).
$$
Up to a simultaneous permutation of the indices, the corresponding
detailed profiles are
$$
    (1,0,0;1,2,2;3),\qquad (1,1,0;1,2,1;3),\qquad (1,1,0;1,2,2;2).
$$

Each profile has an explicit realization with nine regular
values. The regular full-capacity set is open, and the three
aggregate profiles partition it into nonempty open subsets.
Proposition~\ref{prop:profile-transition} shows that a profile can
change while all nine values remain distinct. At the transition one
value loses regularity and its producing support is still unique.
Loss of capacity at a limit is a different phenomenon. For example,
$rA_{153}$ has nine regular values for every $r>0$ by
Proposition~\ref{prop:invariance}, whereas its limit as $r\downarrow0$
has only the Pareto eigenvalue $0$.

The geometry of the three full-capacity profile classes remains
open. In particular, what are their connected components, and which
irregular matrices lie on a boundary between two classes within
$\mathcal F$? It would also be of interest to compare these components
with those of the principally simple and $1$-principally simple
matrices studied in \cite{Kielstra2023}, and to determine which
support profiles can occur at full capacity in higher order.

\section*{Acknowledgements}
The author gratefully acknowledges support from the Math AmSud project N°51756TF (VIPS), ECOS Project C24E06 and the FMJH Gaspard Monge Program for Optimization and Data Science.

\end{document}